\documentclass[11pt]{article}
\usepackage{amsthm, amsmath, amssymb, amsfonts, url, booktabs, tikz, setspace, fancyhdr, bm}
\usepackage{hyperref}
\usepackage{amsthm}
\usepackage{geometry}
\usepackage{hyperref, enumerate}
\usepackage[shortlabels]{enumitem}
\usepackage[babel]{microtype}
\usepackage[english]{babel}
\usepackage[capitalise]{cleveref}
\usepackage{comment}
\usepackage{bbm}
\usepackage{csquotes}
\usepackage{mathabx}
\usepackage{tikz}
\usepackage{graphicx}
\usepackage{float}
\usepackage{xcolor}

\counterwithin{figure}{section}

\newtheorem{theorem}{Theorem}[section]
\newtheorem{prop}[theorem]{Proposition}
\newtheorem{lemma}[theorem]{Lemma}
\newtheorem{cor}[theorem]{Corollary}

\newtheorem{claim}[theorem]{Claim}

\theoremstyle{definition}

\newtheorem*{defn-non}{Definition}

\newtheorem*{rmk}{Remark}

\newlist{Case}{enumerate}{2}
\setlist[Case, 1]{%
    label           =   {\bfseries Case \arabic*.},
    labelindent=1em ,labelwidth=1.3cm, labelsep*=1em, leftmargin =!
}
\setlist[Case, 2]{%
    label           =   {\bfseries Subcase \arabic{Casei}.\arabic*.},
    labelindent=-1em ,labelwidth=1.3cm, labelsep*=1em, leftmargin =!
}

\newenvironment{poc}{\begin{proof}[Proof of claim]}{\end{proof}}

\usepackage{todonotes}

\title{Euclidean Gallai-Ramsey for various configurations}
\author{
Xinbu Cheng\thanks{Laboratory of Mathematics and Complex Systems, Ministry of Education, School of Mathematical Sciences, Beijing Normal University, Beijing, China. Emails: chengxinbu2006@sina.com.}
\and
Zixiang Xu\thanks{Extremal Combinatorics and Probability Group (ECOPRO), Institute for Basic Science (IBS), Daejeon, South Korea. Email: zixiangxu@ibs.re.kr. Supported by IBS-R029-C4.}
}
\date{}

\begin{document}

\maketitle

\begin{abstract}
   The Euclidean Gallai-Ramsey problem, which investigates the existence of monochromatic or rainbow configurations in a colored $n$-dimensional Euclidean space $\mathbb{E}^{n}$, was introduced and studied recently. We further explore this problem for various configurations including triangles, squares, lines, and the structures with specific properties, such as rectangular and spherical configurations. Several of our new results provide refinements to the results presented in a recent work by Mao, Ozeki and Wang. One intriguing phenomenon evident on the Gallai-Ramsey results proven in this paper is that the dimensions of spaces are often independent of the number of colors. Our proofs primarily adopt a geometric perspective. 
\end{abstract}

\section{Introduction}
\subsection{Background}
The Euclidean Ramsey theory, as an important branch of Ramsey theory, was proposed by Erd\H{o}s, Graham, Montgomery, Rothschild, Spencer and Straus~\cite{1973JCTA} in 1975. Originating from graph theory and Ramsey theory, it aims to explore the presence of geometric patterns in continuous domains and investigates the intricate relationship between combinatorial and geometric structures within Euclidean space. Over the past fifty years, a multitude of profound results have been presented in Euclidean Ramsey theory, exemplified by notable works~\cite{2019DCGCONLON,conlon2022more,1990JAMS,2004FranklRodl,1980GrahamJCTA,2009Combinatorica,SHADER1976385} and the references therein, we also refer the readers to the great textbooks~\cite{1997HANDBOOK}. 

In 1967, Gallai~\cite{1967Gallai} proposed the Gallai-Ramsey problem, representing an intriguing branch of graph theory that investigates inherent order and structure within graph coloring. By building upon Ramsey theory and combinatorial principles, it seeks to ascertain the existence of monochromatic and rainbow structures within edge-colored complete graphs. The extremal problem associated with the Gallai-Ramsey problem has garnered significant attention, particularly in recent years, see~\cite{2020SIDMABalogh,2010Survey,2010JGTGyarfas,2021DMLiXihe,2020JGTLiu,2023EuJC}. 

Recently, Mao, Ozeki, and Wang~\cite{2022arxivEGR} combined the concepts of Euclidean Ramsey theory and the Gallai-Ramsey problem, and introduced the Euclidean Gallai-Ramsey problem, which aims to determine the existence of some monochromatic or rainbow configuration in a colored $n$-dimensional Euclidean space $\mathbb{E}^{n}$. Specifically, the problem mainly asks whether there exists an integer $n_{0}$ such that for any $r$-coloring of the points in $\mathbb{E}^{n}$ with $n\geqslant n_{0}$, there is always a monochromatic configuration congruent to $K_{1}$ or a rainbow configuration congruent to $K_{2}$, usually we also say that there is a monochromatic copy of $K_{1}$ or a rainbow copy of $K_{2}$.

In this paper, we further investigate the Euclidean Gallai-Ramsey problem for various configurations. For convenience, we introduce the notation $\mathbb{E}^n\overset{r}{\rightarrow} (K_{1};K_{2})_{\textup{GR}}$, which indicates the existence of either a monochromatic configuration congruent to $K_{1}$ or a rainbow configuration congruent to $K_{2}$ in any $r$-coloring of $\mathbb{E}^{n}$. Conversely, $\mathbb{E}^n\overset{r}{\nrightarrow} (K_{1};K_{2})_{\textup{GR}}$ means that there exists an $r$-coloring of $\mathbb{E}^{n}$ such that there is neither a monochromatic configuration congruent to $K_{1}$ nor a rainbow configuration congruent to $K_{2}$. Obviously, $\mathbb{E}^{n_{0}}\overset{r}{\rightarrow} (K_{1};K_{2})_{\textup{GR}}$ implies $\mathbb{E}^n\overset{r}{\rightarrow} (K_{1};K_{2})_{\textup{GR}}$ for any $n\ge n_{0}$. On the other hand, $\mathbb{E}^n\overset{r_{0}}{\nrightarrow} (K_{1};K_{2})_{\textup{GR}}$ implies $\mathbb{E}^n\overset{r}{\nrightarrow} (K_{1};K_{2})_{\textup{GR}}$ for any $r\ge r_{0}$, as one can only use $r_{0}$ many colors.

Mao, Ozeki and Wang~\cite{2022arxivEGR} initiated the study of this new problem and obtained the following results. 

\begin{theorem}[\cite{2022arxivEGR}]\label{thm:MaoWang}
The followings hold.
    \begin{itemize}
        \item[\textup{(1)}] Let $r$ be a positive integer and $T$ be a triangle with angles $30$, $60$ and $90$ degrees and with hypotenuse of unit length, $\mathbb{E}^3\overset{r}{\rightarrow} (T;T)_{\textup{GR}}$.

        \item[\textup{(2)}] Let $r$ be a positive integer and $Q$ be a rectangle, $\mathbb{E}^{13r+4}\overset{r}{\rightarrow} (Q;Q)_{\textup{GR}}$.

        \item[\textup{(3)}] Let $r\geqslant 3$ be a positive integer and $Q$ be a rectangle with side length $a$ and $b$, where $a\leqslant b\leqslant \sqrt{3}a$, then $\mathbb{E}^{2}\overset{r}{\nrightarrow} (Q;Q)_{\textup{GR}}$.
        \item[\textup{(4)}] Let $r$ be a positive integer, $K_{1}$ be a unit square and $K_{2}$ be an equilateral triangle with side length $1$, then $\mathbb{E}^{r+4}\overset{r}{\rightarrow} (K_{1};K_{2})_{\textup{GR}}$.

         \item[\textup{(5)}] Let $r$ be a positive integer, $K_{1}$ be a triangle with side lengths $1,1,\sqrt{2}$ and $K_{2}$ be an equilateral triangle of side length $1$, then $\mathbb{E}^{r+3}\overset{r}{\rightarrow} (K_{1};K_{2})_{\textup{GR}}$.
    \end{itemize}
\end{theorem}

The proofs of Theorem~\ref{thm:MaoWang}(1)-(3) mainly utilize geometric arguments, while the remaining results are proven by establishing their connection with graph Gallai-Ramsey problems.

\subsection{Our contributions}
The primary objective of the Euclidean Gallai-Ramsey problem is to investigate the conditions under which $\mathbb{E}^n\overset{r}{\rightarrow} (K_{1};K_{2})_{\textup{GR}}$ or $\mathbb{E}^n\overset{r}{\nrightarrow} (K_{1};K_{2})_{\textup{GR}}$ hold for certain configurations $K_{1}, K_{2}$, dimension $n$, and number of colors $r$. We mainly establish a geometric perspective on this problem and obtain several new results, some of which refine the corresponding results in Theorem~\ref{thm:MaoWang}. To aid readers' comprehension, we first introduce relevant notations and definitions that will be utilized later. We will use $\boldsymbol{x}=(x_{1},x_{2},\ldots,x_{n})$ to represent a point in $\mathbb{E}^{n}$. For a pair of points $\boldsymbol{x}$ and $\boldsymbol{y}$, we denote $|\boldsymbol{xy}|$ as the Euclidean distance between $\boldsymbol{x}$ and $\boldsymbol{y}$, and sometimes we use $|\boldsymbol{x}|$ when referring to the distance between $\boldsymbol{x}$ and the origin. A sphere centered at $\boldsymbol{a}=(a_{1},a_{2},\ldots,a_{n})$ with radius $\rho$ is defined as a set of points $\boldsymbol{y}:=(y_{1},y_{2},\ldots,y_{n})$ in $\mathbb{E}^n$ satisfying the equation $\sum_{i=0}^{n}(y_{i}-a_{i})^2=\rho^{2}$. In the context of some affine space in $\mathbb{E}^n$, a sphere refers to the intersection of the sphere and the affine space. We say that a configuration $X$ is \emph{spherical} if it lies on the surface of a sphere. The \emph{circumradius} $\rho(X)$ of a spherical set $X$ represents the smallest radius among all spheres that contain $X$ as a subset. 
We say a simplex $\Delta$ formed by points $\boldsymbol{x}_{1},\boldsymbol{x}_{2},\ldots,\boldsymbol{x}_{k+2}$ with a sequence of heights $h_{1},h_{2},\ldots,h_{k}$, if for each $3\leqslant i\leqslant k+2$, the distance from $\boldsymbol{x}_{i}$ to the affine space spanned by ${\boldsymbol{x}_{1},\boldsymbol{x}_{2},\ldots,\boldsymbol{x}_{i-1}}$ is $h_{i-2}$. For instance, the height of a triangle $\boldsymbol{abc}$ refers to the distance between $\boldsymbol{c}$ and the line $\boldsymbol{ab}$. Furthermore, a configuration $X$ is classified as rectangular if it is a subset of the vertices of a rectangular parallelepiped. We also employ some notations on the classical Euclidean Ramsey problems~\cite{2018Handbook}. For $X\subseteq\mathbb{E}^{n}$, $\mathbb{E}^n\overset{r}{\rightarrow} X$ means for any $r$ coloring of $\mathbb{E}^{n}$, there is a monochromatic configuration congruent to $X$. Let $\mathbb{S}^{n}(h)$ be the $n$-dimensional sphere with radius $h$. Similarly, $\mathbb{S}^{n}(h)\overset{r}{\rightarrow} X$ means that for any $r$ coloring of $\mathbb{S}^{n}(h)$, there is a monochromatic configuration congruent to $X$.

Our first result extends the result in Theorem~\ref{thm:MaoWang}~(1) from a special right triangle to an arbitrary right triangle. 

\begin{theorem}\label{thm:rightTriangle}
Let $T$ be a fixed right triangle. For any positive integer $r$, we have
\begin{equation*}
     \mathbb{E}^{3}\overset{r}{\rightarrow} (T;T)_{\textup{GR}}.
\end{equation*}

\end{theorem}

Our proof of Theorem~\ref{thm:rightTriangle} introduces a more general approach, namely, the rotation method, which can be explored further. Next, we generalize Theorem~\ref{thm:rightTriangle} as follows.

\begin{theorem}\label{thm:off-triangle}
    Let $T_{h}$ be a triangle with some height $h$. Suppose $X\subseteq 
 \mathbb{E}^{n}$ satisfies $\mathbb{S}^{n-2}(h)\overset{2}{\rightarrow} X$,
then for any positive integer $r$, we have
\begin{equation*}
     \mathbb{E}^n\overset{r}{\rightarrow} (X;T_{h})_{\textup{GR}}.
\end{equation*}
\end{theorem}

It is well-known that the condition $\mathbb{S}^{n-2}(h)\overset{2}{\rightarrow} X$ is easily satisfied when $h$ is large, see~\cite{1997HANDBOOK,1995book,nechushtan2002space}. However, when $h$ is relatively small, Lov\'{a}sz~\cite{1983Lovasz} demonstrated that $\mathbb{S}^{n}(h)\overset{n}{\rightarrow} \ell_{2}$ holds for $h> 1/2$, where $\ell_{2}$ is a two-point set which forms a line segment of length $1$. Additionally, Frankl and R\"{o}dl~\cite{1990JAMS} explored the super-Ramsey property and provided results specifically for rectangular sets. Later, Matou\v{s}ek and R\"{o}dl~\cite{MATOUSEK199530} obtained similar results for any simplex. For a comprehensive understanding of these problems, we recommend consulting the excellent textbook~\cite{1997HANDBOOK}.

\begin{theorem}[\cite{1990JAMS,MATOUSEK199530}]\label{thm:RodlFrankl}
The followings hold.
\begin{itemize}
    \item[\textup{(1)}] If $X$ is rectangular and $\rho (X)=\rho $, then for any integer $r$ and real number $\varepsilon > 0$, there exists $n= n(X,r,\varepsilon )$ such that
\begin{equation*}
     \mathbb{S}^{n}(\rho+\varepsilon )\overset{r}{\rightarrow} X.
\end{equation*}

  \item[\textup{(2)}] For any simplex $X$ with $\rho(X)=\rho$, any integer $r$ and any real number $\varepsilon>0$, there exists $n= n(X,r,\varepsilon )$ such that
\begin{equation*}
     \mathbb{S}^{n}(\rho+\varepsilon )\overset{r}{\rightarrow} X.
\end{equation*}
\end{itemize}

\end{theorem}

Combing Theorem~\ref{thm:off-triangle} and Theorem~\ref{thm:RodlFrankl}, we can obtain the following corollary.

\begin{cor}\label{cor:sphere}
 Let $T_{h}$ be a triangle with height $h$ and $X$ be rectangular (or a simplex) with circumradius less than $h$. There exists some $n=n(X,h)$ such that for any positive integer $r$,
 \begin{equation*}
     \mathbb{E}^n\overset{r}{\rightarrow} (X;T_{h})_{\textup{GR}}.
\end{equation*}
In particular, for any triangle $T_h$ with circumradius less than $h$, there exists some $n=n(T_{h})$ such that for any positive integer $r$,
 \begin{equation*}
     \mathbb{E}^n\overset{r}{\rightarrow} (T_{h};T_{h})_{\textup{GR}}.
\end{equation*}

\end{cor}

\begin{rmk}
Comparing with Theorem~\ref{thm:MaoWang}~(4) and~(5), the result in Corollary~\ref{cor:sphere} is more general, for instance, instead of being limited to special triangles or rectangles, we can now set $K_{1}$ to be an arbitrary simplex or rectangle. Furthermore, we can set $K_{2}$ to be a triangle with a height of $h$, rather than being restricted to equilateral triangles. What is even more important is that the dimensions in Theorem~\ref{thm:MaoWang}~(4) and~(5) grow in terms of the number of colors used, while in Corollary~\ref{cor:sphere}, the dimension can be an absolute number that solely depends on the configuration $X$ and the value of $h$. This means that the dimension in Corollary~\ref{cor:sphere} remains constant, regardless of the number of colors being employed.

\end{rmk}

Furthermore, we can extend the scope of our results to encompass general simplices.

\begin{theorem}\label{thm:highsimplex}
Let $r$ and $k$ be positive integers. Let $\Delta $ be a simplex with a sequence of heights $h_{1},h_{2},\ldots,h_{k}$, suppose that $X$ satisfies $ \mathbb{S}^{n-1-j}(h_{j})\overset{j+1}{\rightarrow} X $
for each $1\leq j\leq k$, then we have
\begin{equation*}
    \mathbb{E}^n\overset{r}{\rightarrow} (X;\Delta)_{\textup{GR}}.
\end{equation*}
\end{theorem}

The following diagonal result, which generalizes Corollary~\ref{cor:sphere}, can be deduced from theorem~\ref{thm:RodlFrankl} and theorem~\ref{thm:highsimplex}. For instance, it holds for any regular simplex, as the circumradius of any regular simplex is less than its height.
\begin{cor}\label{cor:simplex}
    For any $(k+1)$-dimensional simplex $\Delta$ with a sequence of heights $h_{1},h_{2},\ldots,h_{k}$, if $\min\limits_{i}h_{i}>\rho(\Delta)$, then there exists some $n_{0}=n_{0}(\Delta)$ such that for any $n\ge n_{0}$ and any positive integer $r$,
    \begin{equation*}
    \mathbb{E}^n\overset{r}{\rightarrow} (\Delta;\Delta)_{\textup{GR}}.
\end{equation*}
\end{cor}

Note that Theorem~\ref{thm:MaoWang}~(2) shows that when $Q$ is a rectangle, then $\mathbb{E}^{13r+4}\overset{r}{\rightarrow}(Q;Q)_{\textup{GR}}$. It is natural to ask whether the results hold when the dimension is independent of the number of colors $r$. Our next result indicates that, when $Q$ is a square, $\mathbb{E}^{n}\overset{r}{\rightarrow}(Q;Q)_{\textup{GR}}$, where $n$ is an absolute constant.

\begin{theorem}\label{thm:UnitSquare}

Let $\square$ be a square. There is a natural number $n_{0}$ such that for any integer $n\geq n_{0}$, and any positive integer $r$, we have 
    \begin{equation*}
   \mathbb{E}^n\overset{r}{\rightarrow} (\square ;\square )_{\textup{GR}}. 
\end{equation*}
\end{theorem}

Write $\ell_{m}$ to denote the set consisting of $m$ points on a line with consecutive points at a distance of one. In the classical Euclidean Ramsey problems, $\ell_{m}$ is one of the most considerable configurations, for example, Erd\H{o}s, Graham, Montgomery, Rothschild, Spencer and Straus~\cite{1973JCTA} provided an explicit $2$-coloring of $\mathbb{E}^{n}$ without monochromatic $\ell_{6}$, and very recently, Conlon and Wu~\cite{conlon2022more} took advantage of a random $2$-coloring of $\mathbb{E}^{n}$ and proved that with positive probability, there is neither a monochromatic $\ell_{3}$ nor a monochromatic $\ell_{10^{50}}$. We say an $r$-coloring of $\mathbb{E}^{n}$ is \emph{spherical coloring} if all points at the same distance from the origin receive the same color. We can see in both of the above celebrated results on Euclidean Ramsey problems for lines, the coloring is spherical. Motivated by this, we prove the following Gallai-Ramsey result for $\ell_{3}$ restricted on spherical coloring. We also show that the similar result under the spherical coloring, does not hold for $\ell_{m}$ with any $m\ge 4$.

\begin{theorem}\label{thm:ell3}
For any positive integer $r$, any positive integer $n\geq 2$, any spherical $r$-coloring of $\mathbb{E}^n$, there is a monochromatic copy of $\ell_{3}$ or a rainbow copy of $\ell_{3}$.
\end{theorem}

Recall that the \emph{diameter} of a finite configuration $X\subseteq\mathbb{E}^{n}$ is defined as the largest distance between any two points in $X$, that is, $\sup\{|\boldsymbol{ab}|:\boldsymbol{a},\boldsymbol{b}\in X\}$. Additionally, the \emph{box-width} of $X\subseteq\mathbb{E}^{n}$ refers to the smallest non-negative real number $a$ such that $X$ can be embedded into $[0,a]\times \mathbb{E}^{n-1}$. The \emph{affine dimension} of an affine space is defined as the dimension of the vector space of its translations. Building upon these definitions, we can derive the following general non-Gallai-Ramsey results by utilizing explicit colorings.

\begin{theorem}\label{thm:coloringGeneral}

Let $X\subseteq\mathbb{E}^{n}$ be a finite configuration with diameter $b$. 

\begin{itemize}
    \item[\textup{(1)}] If the box-width of $X$ is $a>0$ and $P$ is a configuration with $t_{1}$ point diameter at most $a(t_{1}-2)$, then  
\begin{equation*}
    \mathbb{E}^n \overset{\left \lceil b/a \right \rceil+1}{\nrightarrow} (X;P)_{\textup{GR}}.
\end{equation*}

   \item[\textup{(2)}] If $X$ has affine dimension $m<n$, let $g_{X}$ be the minimal number among the diameters of all possible projections of $X$ onto some $(n-m+1)$-dimensional subspaces. Set $h:=\frac{g_{X}}{\sqrt{n-m+1}}$, If $P$ has more than $(t_{2}+1)^{n-m+1}$ points and the diameter of $P$ is at most $ht_{2}$, then 
   \begin{equation*}
    \mathbb{E}^{n} \overset{(\lceil b/h\rceil+1)^{n-m+1}}{\nrightarrow} (X;P)_{\textup{GR}}.
\end{equation*}
\end{itemize}

\end{theorem}

Regarding the result in Theorem~\ref{thm:MaoWang}~(3), the rectangle with side lengths $a$ and $b$, $a\le b\le\sqrt{3}a$ has diameter at most $2a$ and box-width $a$, thus Theorem~\ref{thm:coloringGeneral}~(1) can easily recover this result, and offer much greater flexibility. 

The rest of this paper is organized as follows. We introduce the rotation method and prove Theorems~\ref{thm:rightTriangle},~\ref{thm:off-triangle} and~\ref{thm:highsimplex} in Section~\ref{sec:Rotation}. The proofs of Theorems~\ref{thm:UnitSquare} and~\ref{thm:ell3} leverage some combinatorial ideas, we present them in Section~\ref{sec:MoreCombinatorial}. We give the explicit colorings on points in $\mathbb{E}^{n}$ to demonstrate some non-Gallai-Ramsey results in Section~\ref{sec:nonGallaiRamsey}. Finally we conclude and discuss some further problems in Section~\ref{sec:Remarks}.

\section{The rotation method}\label{sec:Rotation}
In our proofs, we will primarily utilize the rotation method, which revolves around a simple yet powerful idea. In general, to establish $\mathbb{E}^n\overset{r}{\rightarrow} (K_{1};K_{2})_{\textup{GR}}$, we begin by assuming the absence of a rainbow configuration $K_{2}$, based on certain structural properties of $K_{2}$, our objective is then to identify a lower-dimensional space $\mathbb{S}$ where the points in $\mathbb{S}$ can be assigned fewer colors. By leveraging known results pertaining to $\mathbb{S}$, we can enforce the existence of a monochromatic copy of $K_{1}$.

Before delving into the formal proofs, let us introduce two useful facts that will be frequently employed throughout our arguments. Recall that for $X\subseteq\mathbb{E}^{n}$, the notation $\mathbb{E}^n\overset{r}{\rightarrow} X$ means that for any $r$-coloring of $\mathbb{E}^{n}$, there exists a monochromatic configuration congruent to $X$. Regarding right triangles, Shader~\cite{SHADER1976385} established the following result.

\begin{lemma}[\cite{SHADER1976385}]\label{lem:MonoRightTriangle}

Let $T$ be a right triangle, then $\mathbb{E}^2\overset{2}{\rightarrow} T $.

\end{lemma}

\begin{prop}\label{prop:TwoPoint}
Let $n\geqslant 2$ and $r$ be positive integers. If $P\subseteq\mathbb{E}^{n}$ is a two-point set and $X$ be an arbitrary configuration in $\mathbb{E}^n$, then we have 
\begin{equation*}
   \mathbb{E}^n\overset{r}{\rightarrow} (X;P)_{\textup{GR}}. 
\end{equation*}
 
\end{prop}

\begin{proof}
Consider an isometric copy of $P$ denoted by $\boldsymbol{ab}$, where $|\boldsymbol{ab}|$ represents the distance between $\boldsymbol{a}$ and $\boldsymbol{b}$. Let $\boldsymbol{x}_{0}$ be an arbitrary point in $\mathbb{E}^{n}$, assumed to be colored red. Now, for any point $\boldsymbol{x}\in\mathbb{E}^{n}$ with $|\boldsymbol{x}\boldsymbol{x}_{0}|\le 2|\boldsymbol{ab}|$, consider the perpendicular bisector of the line segment $\boldsymbol{x}\boldsymbol{x}_{0}$. On this bisector, choose a point $\boldsymbol{y}$ such that $|\boldsymbol{xy}|=|\boldsymbol{x}_{0}\boldsymbol{y}|=|\boldsymbol{ab}|$. It follows that both $\boldsymbol{y}$ and $\boldsymbol{x}$ must be colored red, which implies that all points within a distance of $2|\boldsymbol{ab}|$ from point $\boldsymbol{x}$ are colored red. Continue this process and finally we conclude that all points in $\mathbb{E}^{n}$ should be colored red, which induces a monochromatic copy of $X$.
\end{proof}

We will first prove the results on right triangles, and then prove more general results in Section~\ref{subsec:general}. 

\subsection{Right triangles: Proof of Theorem~\ref{thm:rightTriangle}}

It is evident that if $\mathbb{E}^{n_{0}}\overset{r}{\rightarrow} (T;T)_{\textup{GR}}$ holds, then $\mathbb{E}^{n}\overset{r}{\rightarrow} (T;T)_{\textup{GR}}$ holds for any $n\ge n_{0}$. Therefore, it suffices to focus on the case when $n=3$. Furthermore, when $r=1$, the problem is trivial, and when $r=2$, we can readily conclude the result by Lemma~\ref{lem:MonoRightTriangle}, thus we can assume that $r\ge 3$. Consider a right triangle $T$ with the lengths of its two legs being $a$ and $b$. Given any $r$-coloring of $\mathbb{E}^{3}$, assume that there is no rainbow configuration congruent to $T$, we Utilize a standard Cartesian coordinate system, and represent each point in $\mathbb{E}^{3}$ using a vector $(x,y,z)$. Since there is no rainbow copy of $T$, our strategy is to identify a two-dimensional space that is colored using only two colors, then we can find a monochromatic right triangle using Lemma~\ref{lem:MonoRightTriangle}.

According to Proposition~\ref{prop:TwoPoint}, and without loss of generality, we can assume that $\boldsymbol{0}:=(0,0,0)$ is colored red, and $\boldsymbol{a}_{1}:=(a,0,0)$ is colored blue. Consider the point $\boldsymbol{b}_{1}:=(0,b,0)$, since $\boldsymbol{0}$, $\boldsymbol{a}_{1}$ and $\boldsymbol{b}_{1}$ form a copy of $T$, $\boldsymbol{b}_{1}$ must be colored either red or blue, otherwise, there exists a rainbow copy of $T$. Moreover, if $\boldsymbol{b}_{1}$ is colored blue, then $\boldsymbol{c}_{1}:=(a,b,0)$ must be colored red. On the other hand, if $\boldsymbol{b}_{1}$ is colored red, then $\boldsymbol{c}_{1}:=(a,b,0)$ must be colored blue. For any real numbers $s$ and $t$ in $\mathbb{R}$, we define $\mathbb{E}^{x=s}:=\{\boldsymbol{v}=(x,y,z)\in\mathbb{E}^{3}:x=s\}$, which represents the subset of $\mathbb{E}^{3}$ consisting of points with $x$-coordinate equal to $s$. Furthermore, we introduce $\mathbb{S}_{\boldsymbol{v},t}^{x=s}$ as the sphere centered at $\boldsymbol{v}$ with radius $t$ within the $2$-dimensional subspace $\mathbb{E}^{x=s}$.

\begin{figure}[htbp]\label{fig:RightT11}
\begin{center}

\begin{tikzpicture}[scale=1.6]

    \fill[red] (-0.8,0) circle (2pt) node[below] {(0,0,0)};
    \fill[blue] (0.8,0) circle (2pt) node[below] {(a,0,0)} ;
    
    \draw (-0.8,0) -- (0.8,0);
    
    \draw[black][dashed] (-0.8,0) ellipse (3/7 and 1);
    \draw[black][dashed] (0.8,0) ellipse (3/7 and 1);
    \node at (-0.8,-1.2) {$\mathbb{S}_{\boldsymbol{0},b}^{x=0}$};
     \node at (0.8,-1.2) {$\mathbb{S}_{\boldsymbol{a}_{1},b}^{x=a}$};
     \fill[black] (-0.8,1) circle (1pt) node[above] {red or blue};
     \draw (-0.8,0) -- (-0.8,1);
     \draw (0.8,0) -- (-0.8,1);

\fill[red] (3.2,0) circle (2pt) node[below] {(0,0,0)};
    \fill[blue] (4.8,0) circle (2pt) node[below] {(a,0,0)} ;

   \draw[black][ dashed] (3.2,0) ellipse (3/7 and 1);
    \draw[black][dashed] (4.8,0) ellipse (3/7 and 1);

\draw (3.2,0) -- (4.8,0);
\draw (3.2,0) -- (2.9,0.75);
     \draw (4.8,0) -- (2.9,0.75);
    
    \node at (3.2,-1.2) {$\mathbb{S}_{\boldsymbol{0},b}^{x=0}$};
     \node at (4.8,-1.2) {$\mathbb{S}_{\boldsymbol{a}_{1},b}^{x=a}$};

    \node at (2.9,2) {$\mathbb{S}_{\boldsymbol{u},b}^{x=0}$};
     \node at (4.5,2) {$\mathbb{S}_{\boldsymbol{w},b}^{x=a}$};

   \node at (2,-2) {All points in dash lines receive red or blue};

   \fill[red] (2.9,0.75) circle (2pt) node[above left]{$u$};

\draw[black][ dashed] (2.9,0.75) ellipse (3/7 and 1);

   \fill[blue] (4.5,0.75) circle (2pt) node [above left] {$w$};

\draw[black][ dashed] (4.5,0.75) ellipse (3/7 and 1);


\end{tikzpicture}
\end{center}
\end{figure}

\begin{figure}[htbp]\label{fig:desk}

    \begin{center}

\begin{tikzpicture}[scale=1.6]

\draw[black][ dashed] (3.2,0) ellipse (6/7 and 2);
\path[fill=gray!30] (3.2,0) ellipse (6/7 and 2);

\fill[red] (3.2,0) circle (2pt) node[below] {(0,0,0)};
    \fill[blue] (4.8,0) circle (2pt) node[below] {(a,0,0)} ;

   \draw[black][ dashed] (3.2,0) ellipse (3/7 and 1);
    \draw[black][dashed] (4.8,0) ellipse (3/7 and 1);

\draw (3.2,0) -- (4.8,0);
\draw (3.2,0) -- (2.9,0.75);
     \draw (4.8,0) -- (2.9,0.75);

     \node at (4.8,-1.2) {$\mathbb{S}_{\boldsymbol{a}_{1},b}^{x=a}$};

    \node at (2.7,2) {$\mathbb{S}_{\boldsymbol{u},b}^{x=0}$};
     \node at (4.5,2) {$\mathbb{S}_{\boldsymbol{w},b}^{x=a}$};
  
   \fill[red] (2.9,0.75) circle (2pt) node[above left]{$u$};

\draw[black][ dashed] (2.9,0.75) ellipse (3/7 and 1);
  
   \fill[blue] (4.5,0.75) circle (2pt) node [above left] {$w$};

\draw[black][ dashed] (4.5,0.75) ellipse (3/7 and 1);

    \draw[black][dashed] (4.8,0) ellipse (3/7 and 1);

\node at (3.2,-1.5) {$\mathbb{D}_{\boldsymbol{0},2b}^{x=0}$};

\end{tikzpicture}

\end{center}
\caption{After repeated rotation, all points on the gray disk receive red or blue }
  \label{fig:myfigure}
\end{figure}

Indeed, each point in the sphere $\mathbb{S}_{\boldsymbol{0},b}^{x=0}$ must be colored either red or blue, because any such point $\boldsymbol{p}$ together with $\boldsymbol{0}$ and $\boldsymbol{a}_{1}$ form a copy of $T$, which cannot be rainbow. Similarly, for each point $(a,y,z)$ in the sphere $\mathbb{S}_{\boldsymbol{a}_{1},b}^{x=a}$, its color should be different from the color of the corresponding point $(0,y,z)\in\mathbb{S}_{\boldsymbol{0},b}^{x=0}$. More specifically, if $(0,y,z)$ in $\mathbb{S}_{\boldsymbol{0},b}^{x=0}$ is colored red (blue), then $(a,y,z)$ in $\mathbb{S}_{\boldsymbol{a}_{1},b}^{x=a}$ is colored blue (red). For every pair of points $\boldsymbol{u}:=(0,y,z)$ and $\boldsymbol{w}:=(a,y,z)$, we can repeat the same procedure by constructing two spheres: $\mathbb{S}_{\boldsymbol{u},b}^{x=0}$ and $\mathbb{S}_{\boldsymbol{w},b}^{x=a}$ centered at $\boldsymbol{u}$ and $\boldsymbol{w}$ respectively. We can see that if a point lie in at least one of the above spheres, it should be colored blue or red. Moreover, observe that any point on the plane $\mathbb{E}^{x=0}$ within a distance of $2b$ from $\boldsymbol{0}$ must be colored red or blue, because it is contained in at least one of the spheres in the collection $\{{\mathbb{S}_{\boldsymbol{u},b}^{x=0}}:{\boldsymbol{u}\in \mathbb{S}_{\boldsymbol{0},b}^{x=0}}\}$ (see, Figure~\ref{fig:desk}). By continuing this operation, we eventually conclude that all points on $\mathbb{E}^{x=0}:=\{\boldsymbol{v}=(x,y,z)\in\mathbb{E}^{3}:x=0\}$ must be colored either red or blue. 

Since $\mathbb{E}^{x=0}$ is a two-dimensional space, according to Lemma~\ref{lem:MonoRightTriangle}, there must exist a monochromatic copy of the right triangle $T$. Thus, the proof is complete.

\subsection{Proofs of the general results}\label{subsec:general}

Here we prove Theorems~\ref{thm:off-triangle} and~\ref{thm:highsimplex} respectively.

\begin{proof}[Proof of Theorem~\ref{thm:off-triangle}]
Let $T_{h}$ be the triangle with points $\boldsymbol{a},\boldsymbol{b},\boldsymbol{c}\in\mathbb{E}^{n}$ satisfying $|\boldsymbol{a}\boldsymbol{b}|=t$ and the distance from the point $\boldsymbol{c}$ to line $\boldsymbol{a}\boldsymbol{b}$ is $h$. Then we can find a point $\boldsymbol{d}$ in line $\boldsymbol{a}\boldsymbol{b}$ such as $|\boldsymbol{c}\boldsymbol{d}|=h$ and $\boldsymbol{cd}$ orthogonal to $\boldsymbol{ab}$. For any $r$-coloring of $\mathbb{E}^{n}$, suppose that there is no rainbow copy of $T_{h}$, then by Proposition~\ref{prop:TwoPoint} and without loss of generality, we can first assume $\boldsymbol{a}=(0,0,\ldots,0)$ and $\boldsymbol{b}=(t,0,\ldots,0)$ are colored red and blue. Next, we rotate $\boldsymbol{c}$ with respect to $\boldsymbol{ab}$, more precisely, if $\boldsymbol{d}=(|\boldsymbol{ad}|,0,\ldots,0)$, we consider the sphere $\mathbb{S}^{n-2}(h)\subseteq  \mathbb{E}^{x=|\boldsymbol{ad}|}$ centered at $\boldsymbol{d}\in\mathbb{E}^{n}$ and with radius $h$. Note that the points in this sphere should be colored red or blue, otherwise, we can find a rainbow triangle $T_{h}$. Thus, by the condition $\mathbb{S}^{n-2}(h)\overset{2}{\rightarrow} X$, we can find a monochromatic copy of $X$ in $\mathbb{S}^{n-2}(h)\subseteq \mathbb{E}^{x=|\boldsymbol{ad}|}$. If $\boldsymbol{d}=(-|\boldsymbol{ad}|,0,\ldots,0)$, we can use a similar argument and omit the repeated details.  
\end{proof}

\begin{proof}[Proof of Theorem~\ref{thm:highsimplex}]
For any $r$-coloring $\chi$ of $\mathbb{E}^{n}$, suppose there is no monochromatic $X$, our goal is to find a rainbow copy of configuration $\Delta$. In this case, our strategy is to build a rainbow copy of $\Delta$ iteratively. First, by Proposition~\ref{prop:TwoPoint}, there is a two-point set $\{\boldsymbol{x}_{1}',\boldsymbol{x}_{2}'\}$ in $\mathbb{E}^{n}$, which is isometric to $\{\boldsymbol{x}_{1},\boldsymbol{x}_{2}\}\subseteq \Delta$. Fix the set $\{\boldsymbol{x}_{1}',\boldsymbol{x}_{2}'\}$, we then consider the set of points $Y_{3}:=\{\boldsymbol{y}\in\mathbb{E}^{n}:|\boldsymbol{x}_{1}'\boldsymbol{y}|=|\boldsymbol{x}_{1}\boldsymbol{x}_{3}|,|\boldsymbol{x}_{2}'\boldsymbol{y}|=|\boldsymbol{x}_{2}\boldsymbol{x}_{3}|\}$, which can be viewed as an intersection of two spheres. Note that, indeed $Y_{3}$ itself is an $(n-2)$-dimensional sphere with radius $h_{1}$ centered at some point on line $\boldsymbol{x}_{1}\boldsymbol{x}_{2}$. Now, suppose that all points in $Y_{3}$ receive either $\chi(\boldsymbol{x}_{1})$ or $\chi(\boldsymbol{x}_{2})$, then the condition $\mathbb{S}^{n-2}(h_{1})\overset{2}{\rightarrow} X$ implies the existence of a monochromatic copy of $X$, a contradiction. Therefore, there must be some point, namely $\boldsymbol{x}_{3}'\in Y_{3}$, such that $\chi(\boldsymbol{x}_{3}')\notin\{\chi(\boldsymbol{x}_{1}'),\chi(\boldsymbol{x}_{2}')\}$. Next, fix the set $\{\boldsymbol{x}_{1}',\boldsymbol{x}_{2}',\boldsymbol{x}_{3}'\}$, we can use the similar argument to find a point $\boldsymbol{x}_{4}'$ such that the distance from $\boldsymbol{x}_{4}'$ to the affine space spanned by $\{\boldsymbol{x}_{1}',\boldsymbol{x}_{2}',\boldsymbol{x}_{3}'\}$ is $h_{2}$ , $\chi(\boldsymbol{x}_{4}')\notin\{\chi(\boldsymbol{x}_{1}'),\chi(\boldsymbol{x}_{2}'),\chi(\boldsymbol{x}_{3}')\}$ and $|\boldsymbol{x}_{4}'\boldsymbol{x}_{i}'|=|\boldsymbol{x}_{4}\boldsymbol{x}_{i}|$ for each $1\le i\le 3$. Continue this process and repeatedly take advantage of the conditions that $ \mathbb{S}^{n-1-j}(h_{j})\overset{j+1}{\rightarrow} X $
for each $1\leq j\leq k$, eventually we can find a rainbow isometric copy of $\Delta$ in $\mathbb{E}^{n}$, as desired.

\end{proof}

\section{Proofs of Theorems~\ref{thm:UnitSquare} and~\ref{thm:ell3} }\label{sec:MoreCombinatorial}
The previous proofs are mainly geometric, here we mix some more combinatorial ideas to show the results on the squares and the lines.

\subsection{Squares: Proof of Theorem~\ref{thm:UnitSquare}}
For any $r$-coloring $\chi$ of points in $\mathbb{E}^n$, it suffices to prove the case of the unit square since the space can be stretched and compressed. In this proof, we mainly take advantage of a special configuration in a $5$-dimensional space, more precisely, let $Q_{5}:=\frac{1}{\sqrt{2}}\{0,1\}^5=\{0,\frac{1}{\sqrt{2}}\}^5$, consists of all vectors of length $5$ with entries $0$ and $\frac{1}{\sqrt{2}}$. Note that $Q_{5}$ is rectangular with $\rho(Q_{5})=\frac{\sqrt5}{\sqrt8}<1$.

 We first assume that there is neither a monochromatic copy of the unit square 
 nor a rainbow copy of the unit square.
 By Proposition~\ref{prop:TwoPoint} and without loss of generality, we can assume that the points $(0,0,\ldots,0)$ and $(1,0,\ldots,0)$ in $\mathbb{E}^{n}$ are colored red and blue respectively. For any $\boldsymbol{x}\in \mathbb{E}^{n-1}$ with $|\boldsymbol{x}|=1$, one can easily check that the four points $(0,0,\ldots,0),(1,0,\ldots,0),(0,\boldsymbol{x}),(1,\boldsymbol{x})$ in $\mathbb{E}^{n}$ together form a copy of unit square. By our assumption, we know that this unit square is not rainbow, thus regarding the colors of points $(0,\boldsymbol{x})$ and $(1,\boldsymbol{x})$ in $\mathbb{E}^{n}$ under the coloring function $\chi$, there are at most five types. We list them as follows:
  \begin{itemize}
 \item type $1$, 
  $(0,\boldsymbol{x})$ is red
   \item type $2$, $(0,\boldsymbol{x})$ is blue
   \item type $3$, $(1,\boldsymbol{x})$ is red
   \item type $4$, $(1,\boldsymbol{x})$ is blue
   \item type $5$, 
 $(0,\boldsymbol{x})$ and 
   $(1,\boldsymbol{x})$ receive the same color
   
 \end{itemize}
Note that, several types can occur simultaneously, for example, both $(0,\boldsymbol{x})$ and $(1,\boldsymbol{x})$ can be colored red, which indicates that types $1$, $3$ and $5$ occur simultaneously.
 Next, we focus on the $(n-2)$-dimensional sphere $\mathbb{S}^{n-2}(1)$ with radius $1$, and design an auxiliary $5$-coloring function $\gamma$ on $\mathbb{S}^{n-2}(1)$ according to the types. More precisely, for each point $\boldsymbol{x} \in \mathbb{S}^{n-2}(1)$, $\gamma(\boldsymbol{x})=c_{i}$ if the pair of vectors $(0,\boldsymbol{x})$ and $(1,\boldsymbol{x})$ satisfies the color type $i$ in the above list. In particular, if more than one types occur, we can arbitrarily choose one of them. Since $\rho(Q_{5})=\sqrt{\frac{5}{8}}<1$, by Theorem~\ref{thm:RodlFrankl}, there is a monochromatic $Q_{5}$ in $\mathbb{S}^{n-2}(1)$ under coloring function $\gamma$ if $n>C$ for some positive integer $C$.

Suppose that there is a monochromatic copy of $Q_{5}$ whose points are colored $c_{1}$, we can pick four points in this $Q_{5}$, namely, $\frac{1}{\sqrt{2}}(1,1,1,0,0)$, $\frac{1}{\sqrt{2}}(1,0,1,1,0)$, $\frac{1}{\sqrt{2}}(1,0,0,1,1)$, $\frac{1}{\sqrt{2}}(1,1,0,0,1)$. One can easily check that these points form a copy of a monochromatic unit square under the coloring $\gamma$. By definition of $\gamma$, this induces a red copy of unit square on the hyperplane $x_{1}=0$ under the coloring function $\chi$, a contradiction. Similar arguments hold when there is a monochromatic copy of $Q_{5}$ whose points are colored $c_{2}$, $c_{3}$ or $c_{4}$ under the coloring $\gamma$, we omit the repeated details.

Therefore, we can assume that there is a monochromatic copy of $Q_{5}$ whose points are colored $c_{5}$. Our strategy is designing a new coloring function $\delta$ on $Q_{5}$ in some $5$-dimensional subspace according to the original coloring function $\chi$ of $\mathbb{E}^{n}$. More precisely, let $Q\subseteq\mathbb{S}^{n-2}(1)$ be a monochromatic isometric copy of $Q_{5}$ under the isometric mapping $\Phi: Q_{5}\rightarrow Q$, then for any $\boldsymbol{x}\in Q\subseteq\mathbb{S}^{n-2}(1)$, if the points $(0,\boldsymbol{x})$ and $(1,\boldsymbol{x})$ in $\mathbb{E}^{n}$ receive color $t_{i}$ under coloring function $\chi$, then we color the point $\boldsymbol{x}':=\Phi^{-1}(\boldsymbol{x})$ by $\delta(\boldsymbol{x}')=t_{i}$. By the definition of the coloring function $\delta$, if any of the following cases occur, then our objective is achieved, as we can find a monochromatic or rainbow copy of a unit square in $\mathbb{E}^{n}$ under the original coloring function $\chi$.
 \begin{Case}
     \item There are two points $\boldsymbol{x},\boldsymbol{y}\in Q_{5}$ with $|\boldsymbol{xy}|=1$ receiving same color under the function $\delta$. In this case, four points $(0,\Phi(\boldsymbol{x})), (1,\Phi(\boldsymbol{x})), (0,\Phi(\boldsymbol{y})),(1,\Phi(\boldsymbol{y}))$ together form a monochromatic copy of unit square in $\mathbb{E}^{n}$ under the coloring function $\chi$.

     \item There are four points $\boldsymbol{a}$, $\boldsymbol{b}$, $\boldsymbol{c}$, and $\boldsymbol{d}$ forming a rainbow copy of unit square in $Q_{5}$ under the coloring function $\delta$. In this case, four points    
     $(0,\Phi(\boldsymbol{a})),(0,\Phi(\boldsymbol{b})),(0,\Phi(\boldsymbol{c})),(0,\Phi(\boldsymbol{d}))$ together form a rainbow copy of unit square in $\mathbb{E}^{n}$ under the coloring function $\chi$.   
 \end{Case}

It suffices to prove that, for any coloring function $\delta$ on $Q_{5}$, at least one of the above cases occurs, to achieve this, we can only consider an even smaller subset of $Q_{5}$. More precisely, let $Q_{5}^{(3)}\subseteq Q_{5}$ be a collection of vectors in $Q_{5}$ with exactly $3$ non-zero entries. For convenience, we will use a sequence of length $3$ to represent each member in $Q_{5}^{(3)}$, for example, we use $135$ to represent $\frac{1}{\sqrt{2}}(1,0,1,0,1)$, where three entries $1$ are in the positions $1$, $3$ and $5$. Note that the Euclidean distance between any two distinct points in $Q_{5}$ is determined by the number of positions at which the corresponding symbols are distinct. For example, $123$ and $124$ are at distance $1$ since there are two such positions between them, and $123$ and $145$ are at distance $\sqrt{2}$ since there are four such positions.

Suppose that under the coloring function $\delta$ on $Q_{5}$, none of Case 1 or Case $2$ occurs, we can see that the set of points $\{123,234,124,134\}$ form a copy of regular tetrahedron with side length one, thus $\delta(123)$, $\delta(234)$, $\delta(123)$ and $\delta(134)$ are pairwise distinct, otherwise Case 1 occurs. Furthermore, we consider a copy of the unit square formed by the set of points $\{123,125,234,245\}$. This unit square cannot be rainbow, otherwise, Case 2 occurs, moreover, $\delta(125)\neq \delta(245)$, $\delta(125)\neq\delta(123)$, $\delta(123)\neq\delta(234)$ and $\delta(234)\neq\delta(245)$, otherwise Case 1 occurs. Thus, at least one of $\delta(125)=\delta(234)$ and $\delta(123)=\delta(245)$ shall occur. By symmetry, we can assume that $\delta(125)=\delta(234)$. We consider another unit square formed by $\{124,125,134,135\}$, similarly, this unit square is also not rainbow, which implies that at least one of the cases $\delta(124)=\delta(135)$ and $\delta(125)=\delta(134)$ shall occur. Suppose that $\delta(125)=\delta(134)$, then $\delta(125)\neq\delta(234)$ as $\delta(134)\neq\delta(234)$, a contradiction, thus we have $\delta(124)=\delta(135)$. Continue by considering the unit square formed by $\{123,135,234,345\}$, similarly we can see at least one of the cases $\delta(123)=\delta(345)$ and $\delta(135)=\delta(234)$ shall occur. Suppose that $\delta(135)=\delta(234)$, then we have $\delta(135)\neq\delta(124)$ as $\delta(124)\neq\delta(234)$, a contradiction, therefore we have $\delta(123)=\delta(345)$. 

The final unit square we shall consider is formed by $\{124,134,245,345\}$, for the similar reasons, at one of the cases $\delta(134)=\delta(245)$ and $\delta(124)=\delta(345)$ should occur. However, it is impossible that $\delta(124)=\delta(345)$ because we already show that $\delta(124)\neq\delta(123)$ and $\delta(123)=\delta(345)$. By the above analysis, as $\delta(125)\neq\delta(123)=\delta(345)$ and $\delta(135)\neq\delta(134)=\delta(245)$, then we can find a rainbow copy of unit square formed by $\{125,135,245,345\}$, the proof is finished.

\subsection{Lines: Proof of Theorem~\ref{thm:ell3}}
Assume $\chi_{0}$ is a spherical coloring of $\mathbb{E}^n$, we are going to find a monochromatic copy of $\ell_{3}$ or a rainbow copy $\ell_{3}$ in $\mathbb{E}^n$. Firstly we define a new coloring $\chi$ of the set of non-negative real numbers $\mathbb{R}^{\geq 0}$ as $\chi(|\boldsymbol{x}|^2)=\chi_{0}(\boldsymbol{x})$ for all $\boldsymbol{x}\in \mathbb{E}^n$, where $|\boldsymbol{x}|$ is the distance between $\boldsymbol{x}$ and $\boldsymbol{0}$. Then we call a triple  $({y}_{1},{y}_{2},{y}_{3})\in \mathbb{R}^{\geq 0}\times \mathbb{R}^{\geq 0}\times \mathbb{R}^{\geq 0}$ is \emph{potential} if there is a line $\ell_{3}$ formed by $\boldsymbol{x}_{1},\boldsymbol{x}_{2},\boldsymbol{x}_{3}$ such that $|\boldsymbol{x}_{i}|^2=y_{i}$ for $1\leq i\leq 3$.  

Motivated by a recent work of Conlon and Wu~\cite{conlon2022more}, we can see that if $\boldsymbol{x}_{1},\boldsymbol{x}_{2},\boldsymbol{x}_{3}\in\mathbb{E}^{n}$ form a copy of $\ell_{3}$, then we have $|\boldsymbol{x}_{1}|^{2}=|\boldsymbol{x}_{2}|^{2}+1-2|\boldsymbol{x}_{2}|\cos{\alpha}$ and $|\boldsymbol{x}_{3}|^{2}=|\boldsymbol{x}_{2}|^{2}+1+2|\boldsymbol{x}_{2}|\cos{\alpha}$, where $\alpha$ is the angle $\boldsymbol{x}_{1}\boldsymbol{x}_{2}\boldsymbol{0}$. Those together imply that $|\boldsymbol{x}_{1}|^{2}+|\boldsymbol{x}_{3}|^{2}=2|\boldsymbol{x}_{2}|^{2}+2$.
Hence, if ${y_{1},y_{2},y_{3}}$ satisfy $y_{1} + y_{3} = 2y_{2} + 2$ and $\sqrt{y_{2}} \ge 2\max\{|y_{1}-y_{2}|,|y_{2}-y_{3}|,1\}$, then $(y_{1},y_{2},y_{3})$ is a potential triple. Moreover, as the coloring $\chi_{0}$ is a spherical coloring, then the existence of a monochromatic or rainbow potential triple $(y_{1},y_{2},y_{3})$ yields the existence of a monochromatic or rainbow copy of $\ell_{3}$. To see this, suppose such a triple $(y_{1},y_{2},y_{3})$ exists, then we can find a corresponding triple of points $(\boldsymbol{x}_{1},\boldsymbol{x}_{2},\boldsymbol{x}_{3})\in\mathbb{E}^{n}\times\mathbb{E}^{n}\times\mathbb{E}^{n}$, where $\boldsymbol{x}_{1}=(\sqrt{y_{2}}-\cos{\alpha},\sin{\alpha},0,\ldots,0)$, $\boldsymbol{x}_{2}=(\sqrt{y_{2}},0,0,\ldots,0)$ and $\boldsymbol{x}_{3}=(\sqrt{y_{2}}+\cos{\alpha}, -\sin{\alpha},0,\ldots,0)$. Therefore, it suffices to show the existence of a monochromatic or rainbow potential triple $(y_{1},y_{2},y_{3})$ under the coloring function $\chi$. 

Next, suppose there is neither a rainbow copy of $\ell_{3}$ nor a monochromatic copy of $\ell_{3}$, let $N\in \mathbb{Z}$ be a large enough integer and $a_{1}:=\chi (N)$, since $(N+1,N,N+1)$ is a potential triple, $a_{2}:=\chi(N+1)\neq\chi(N)=a_{1}$. Furthermore, $(N,N,N+2)$ and $(N+2,N+1,N+2)$ are two potential triples, hence $a_{3}:=\chi(N+2)$ should be distinct from $a_{1}$ and $a_{2}$. Therefore, if $r=2$, we are already done. Next, we assume that $r\ge 3$ and we have the following claim.

\begin{claim}\label{claim:mod}
For any integer $M\ge N$, we have $\chi(M)=a_{M-N+1 \pmod{3}}$.
\end{claim}
\begin{poc}
We consider $\chi(N+3)$ and two other potential triples, namely, $(N+1,N+1/2,N+2)$ and $(N,N+1/2,N+3)$. Note that if $a_{4}:=\chi(N+3)\notin\{a_{1},a_{2},a_{3}\}$, then $\chi (N+1/2)$ satisfies either $\chi(N+1/2)\notin\{a_{1},a_{2}\}$ or $\chi(N+1/2)\notin\{a_{3},a_{4}\}$, which yields at least one of $(N+1,N+1/2,N+2)$ and $(N,N+1/2,N+3)$ is rainbow, a contradiction. Hence, $a_{4}\in\{a_{1},a_{2},a_{3}\}$. Moreover, we can see that $a_{4}\notin\{a_{2},a_{3}\}$, otherwise $(N+1,N+1/2,N+2)$ forms a rainbow potential triple, thus $a_{4}=a_{1}$. One can run a similar argument to complete the proof, we omit the repeated details.
\end{poc}
 We further consider the potential triples $(N+1/3,N+1+1/3,N+2+1/3)$ and $(N+2/3,N+1+2/3,N+2+2/3)$, and their colorings $(b_{1},b_{2},b_{3}):=(\chi(N+1/3),\chi(N+1+1/3),\chi(N+2+1/3))$ and $(c_{1},c_{2},c_{3}):=(\chi(N+2/3),\chi(N+1+2/3),\chi(N+2+2/3))$, for the same reasons as above, we have $b_{i}\neq b_{j}$ and $c_{i}\neq c_{j}$ for $i\neq j$. Similar as Claim~\ref{claim:mod}, for any integer $M\ge N$, we have $\chi(M+1/3)=b_{M-N+1 \pmod{3}}$ and $\chi(M+2/3)=c_{M-N+1 \pmod{3}}$.

Note that, all of the triples $(N+1/3,N,N+1+2/3)$, $(N,N+1/3,N+2+2/3)$ and $(N,N+2/3,N+3+1/3)$ are potential, the corresponding colors are $(b_{1},a_{1},c_{2})$, $(a_{1},b_{1},c_{3})$ and $(a_{1},c_{1},b_{1})$, respectively. As there is no rainbow copy of $\ell_{3}$, then we have $a_{1}=b_{1}$. Furthermore, all of $(N+2/3,N,N+1+1/3)$, $(N+3,N+2/3,N+1/3)$ and $(N+6,N+2+1/3,N+2/3)$ are potential, since there is no rainbow copy of $\ell_{3}$, similarly, we have $a_{1}=c_{1}$. Then the potential triple $(N,N+2/3,N+3+1/3)$ yields a monochromatic copy of $\ell_{3}$, a contradiction. Then the proof is finished.

\begin{center}
\begin{figure}

\begin{tikzpicture}[scale=0.8]
    \draw[-] (-10,0) -- (10,0);
    
    \definecolor{red}{RGB}{255,0,0}
    \definecolor{blue}{RGB}{0,0,255}
    \definecolor{green}{RGB}{0,128,0}
    
    \foreach \x/\color/\label in {-9/red/N, -7/blue/N+1, -5/green/N+2, -3/red/N+3, -1/blue/N+4, 1/green/N+5, 3/red/N+6, 5/blue/N+7, 7/green/N+8, 9/red/N+9} {
        \node[circle, fill=\color, inner sep=2.5pt] at (\x,0) {};
        \node[below] at (\x,0) {\label};

    }
     \draw (-7.66,0) circle (0.1) node [above]{$c_{1}$};
     \fill[red](-7.66,0) circle (0.15);
     \draw (-5.66,0) circle (0.1)node [above]{$c_{2}$};
     \draw (-3.66,0) circle (0.1)node [above]{$c_{3}$};
     \draw (-1.66,0) circle (0.1)node [above]{$c_{1}$};
     \draw (0.33,0) circle (0.1);
     \draw (2.33,0) circle (0.1);
     \draw (4.33,0) circle (0.1);
     \draw (6.33,0) circle (0.1);
     \draw (8.33,0) circle (0.1);

     \draw (-8.33,0) circle (0.1)node [above]{$b_{1}$};
     
     \draw (-6.33,0) circle (0.1)node [above]{$b_{2}$};
     \draw (-4.33,0) circle (0.1)node [above]{$b_{3}$};
     \draw (-2.33,0) circle (0.1)node [above]{$b_{1}$};
\fill[red](-2.33,0) circle (0.15);
     
     \draw (-0.33,0) circle (0.1);
     \draw (1.66,0) circle (0.1);
     \draw (3.66,0) circle (0.1);
     \draw (5.66,0) circle (0.1);
     \draw (7.66,0) circle (0.1);

\end{tikzpicture}
\caption{By our analysis, $N$, $N+2/3$ and $N+3+1/3$ should receive the same color}
\end{figure}
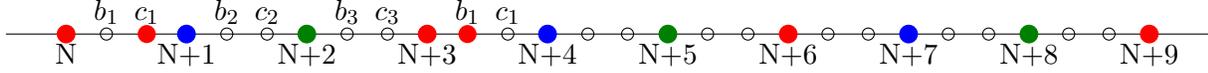

\end{center}
\begin{rmk}
Here we simply explain that the similar result restricted on the spherical coloring as Theorem~\ref{thm:ell3}, does not hold for $\ell_{m}$ with any $m\ge 4$.
Indeed, Erd\H{o}s, Graham, Montgomery, Rothschild, Spencer and Straus~\cite{1973JCTA} provided an explicit coloring function on $\mathbb{E}^{n}$ with $3$ colors to show that $\mathbb{E}^n \overset{r=3}{\nrightarrow} \ell_{4}$. Note that $\ell_{4}$ contains $4$ points, thus if $r\ge 3$, we have $\mathbb{E}^n \overset{r}{\nrightarrow} (\ell_{4};\ell_{4})$. Moreover, there is another spherical coloring to $\mathbb{E}^n$ by coloring the points $\boldsymbol{x}\in \mathbb{E}^n$ with color $\left \lfloor |\boldsymbol{x}|^2\right \rfloor$ (mod $4$) giving that $\mathbb{E}^n \overset{r=4}{\nrightarrow} \ell_{3}$. This implies, restricted on spherical coloring, $\mathbb{E}^n \overset{r\geq 4}{\nrightarrow} (\ell_{3};P)$ holds for any configuration $P$ with at least $5$ points.  
\end{rmk}

\section{Non-Gallai-Ramsey: Proof of Theorem~\ref{thm:coloringGeneral}}\label{sec:nonGallaiRamsey}

Suppose $X$ has positive box-width $a$, then we first divide $\mathbb{E}^{n}$ into several blocks, namely, $\mathcal{B}_{i}:=[(i-1)a,ia)\times\mathbb{E}^{n-1}$, where $i\in\mathbb{Z}$. Then we color all points in block $\mathcal{B}_{i}$ using color $j$, where $j\equiv i\pmod{(\lceil \frac{b}{a}\rceil+1)}$. It remains to show that under this $(\lceil \frac{b}{a}\rceil+1)$-coloring of $\mathbb{E}^{n}$, there is neither a monochromatic copy of $X$ nor a rainbow copy of $P$. First, note that $X$ cannot be embedded into any single block $[(i-1)a,ia)\times\mathbb{E}^{n-1}$ as the box-width of $X$ is exactly $a$. Thus, any monochromatic copy of $X$ should overlap at least two distinct blocks with the same color. However, it is impossible, to see this, suppose there is some monochromatic copy of $X$ overlapping $\mathcal{B}_{1}$ and $\mathcal{B}_{2}$, by definition of the coloring function, any pair of points $\boldsymbol{b}_{1}\in\mathcal{B}_{1}$ and $\boldsymbol{b}_{2}\in\mathcal{B}_{2}$ are at a distance larger than $b$, but the diameter of the configuration $X$ is exactly $b$. Similarly, suppose that there is a rainbow copy of $P$ with $t_{1}$ points, $P$ must overlap at least $t_{1}$ many distinct blocks to receiving $t_{1}$ distinct colors, which is also impossible because the diameter of $P$ is at most $(t_{1}-2)a$. The proof of the first part then is finished.

If the box-width of $X$ is $0$, and the affine dimension of $X$ is $m<n$, then $X$ contains an affine hull with dimension $m<n$. Since $X$ is a finite configuration with affine dimension $m$, then for any $(n-m+1)$-dimensional subspace $V$ of $X$, the projection of $X$ on the subspace $V$ has a positive diameter. 

Let $g_{X}>0$ be the minimal diameter among all possible projections of $X$ on $(n-m+1)$-dimensional subspaces. Note that such minimal diameter $g_{X}$ always exists, since the set of $(n-m+1)$-dimensional subspaces of $\mathbb{E}^{n}$ is compact (also called Grassmannian, see~\cite{Lee}), and the image on $\mathbb{E}$ of a compact set under continues map always contains a minimal element. 

Then it suffices to carefully construct a collection of blocks according to the structural properties of the projection of $X$ on $V$.
 Set $h:=\frac{g_{X}}{\sqrt{n-m+1}}$, we can construct a collection of blocks $\{\mathcal{C}_{(i_{1},i_{2},\ldots,i_{n-m+1})}\}_{(i_{1},i_{2},\ldots,i_{n-m+1})\in\mathbb{Z}^{n-m+1}}$ as 
\begin{equation*}
    \mathcal{C}_{(i_{1},i_{2},\ldots,i_{n-m+1})}:=[(i_{1}-1)h,i_{1}h)\times\cdots\times [(i_{n-m+1}-1)h,i_{n-m+1}h)\times\mathbb{E}^{m-1}.
\end{equation*}
Then we color all points in $\mathcal{C}_{(i_{1},i_{2},\ldots,i_{n-m+1})}$ by colors $(j_{1},j_{2},\ldots,j_{n-m+1})$, where $j_{s}\equiv i_{s}\pmod{(\lceil \frac{b}{h}\rceil+1)}$.

We claim that for any monochromatic isometric copy of $X$, it overlaps at least two distinct blocks. To see this, consider an arbitrary projection of $X$ on an $(n-m+1)$-dimensional space $V$, the diameter of the projection of $X$ on $V$ is at least $g_{X}$. However, for any pairs of points $(\boldsymbol{p}_{1},\boldsymbol{p}_{2})$ in the same block, the distance of the corresponding pair of points on the subspace $V$ is smaller than $g_{X}$. By definition of the collection of blocks and the colorings, we can see the distance of any pair of points receiving the same color in distinct blocks is larger than $b$, therefore, there is no monochromatic copy of $X$.

Then we turn to consider the configuration $P$. Suppose that there is a rainbow copy of $P$, we claim that the number of distinct indices $i_{1}$ in all blocks that $P$ overlaps are at most $t_{2}+1$, otherwise there are two points in this rainbow copy of $P$ whose distance is larger than $t_{2}h$, a contradiction. This claim obviously holds for any position $i_{s}$, $2\le s\le n-m+1$. Therefore, the total number of possible distinct sequences of indices $(i_{1},i_{2},\ldots,i_{n-m+1})$ of blocks that $P$ overlaps is at most $(t_{2}+1)^{n-m+1}$, which is smaller than the number of points of $P$, yields there is no rainbow copy of $P$. The proof is finished.

\section{Concluding remarks}\label{sec:Remarks}
Our preliminary goals are to strengthen and extend the results in~\cite{2022arxivEGR}, as outlined in Theorem 1.1. We achieve partial progress toward these goals and our proofs mix some geometric observations and operations such as the rotation method and some combinatorial ideas. We believe that this framework can be further leveraged and extended to this problem and its related variations. For example, as pointed out by Debsoumya Chakraborti and Minho Cho (private communication), the method can be applied to find some monochromatic high dimensional spheres when we forbid some certain rainbow configurations. In conclusion, we provide a comprehensive summary of our new results and propose additional avenues for exploration.
\begin{itemize}
    \item We show that for any right triangle $T$, positive integers $n\ge 3$ and $r$, we have $\mathbb{E}^{n}\overset{r}{\rightarrow} (T;T)_{\textup{GR}}$, which extends the result in Theorem~\ref{thm:MaoWang}~(1) on a very special right triangle. Furthermore, we show that for general triangle $T$, whose height is larger than circumradius, $\mathbb{E}^{n}\overset{r}{\rightarrow} (T;T)_{\textup{GR}}$ holds for $n\ge n(T)$. We wonder whether similar results hold in $3$-dimensional Euclidean space for any fixed triangle.
\item Note that in Theorem~\ref{thm:MaoWang}~(2), $n$ grows linearly in the number of colors $r$, we suspect that for any rectangle $Q$, $\mathbb{E}^n\overset{r}{\rightarrow} (Q;Q )_{\textup{GR}}$ holds, where $n$ is independent of $r$. In particular, We prove that for any positive integer $r$ and any square $\square$, there is some absolute number $n$ such that $\mathbb{E}^n\overset{r}{\rightarrow} (\square ;\square )_{\textup{GR}}$, which perhaps provides some evidence to our conjecture. 

\item Theorem~\ref{thm:MaoWang}~(3) shows some non-Gallai-Ramsey results on certain rectangles, we extend this result to a much larger family of configurations in Theorem~\ref{thm:coloringGeneral}. Our coloring functions there are explicit, and we think it would be interesting to seek more non-Gallai-Ramsey typle results via mixing some ideas from both probabilistic and deterministic approaches.

\item Theorem~\ref{thm:MaoWang}~(4) and (5) present some off-diagonal Gallai-Ramsey results, the proofs mainly rely on some known results on graph Gallai-Ramsey problems. We also extend it to a more flexible form in Corollary~\ref{cor:sphere}, in particular, the corresponding dimensions in our results could be an absolute number. Improving the quantitative bound on the dimensions is also of interest.
\end{itemize}

\section*{Acknowledgement}
Zixiang Xu would like to thank Debsoumya Chakraborti and Minho Cho for helpful discussions and thank Xihe Li for introducing some basics on Gallai-Ramsey type problems.

\bibliographystyle{abbrv}
\bibliography{gallairamsey}
\end{document}